\documentclass[11pt,a4paper]{article}
\usepackage{amssymb,latexsym}
\usepackage{bm, amsmath, amssymb, amsthm}

\newtheorem{theorem}{Theorem}%[section]
\newtheorem{corollary}{Corollary}%[section]
\newtheorem{lemma}{Lemma}%[section] %[theorem]
\newtheorem{proposition}{Proposition}%[section]
\newtheorem{remark}{Remark}%[section]

\def\RR{\mathbb{R}}
\def\Ric{\operatorname{Ric}}
\def\tr{\operatorname{Tr\,}}
\def\Div{\operatorname{div}}
\def\vol{\operatorname{vol}}

\title{From the Ricci flow evolution equation to vanishing theorems for Ricci solitons}

\author{Vladimir Rovenski\footnote{Department of Mathematics, University of Haifa, Mount Carmel, Haifa, 31905, Israel.
E-mail: vrovenski@univ.haifa.ac.il},
\ Sergey Stepanov\footnote{Department of Mathematics, Russian Institute for Scientific
and Technical Information of the Russian Academy of Sciences, 20, Usievicha street, 125190 Moscow, Russia.
E-mail: s.e.stepanov@mail.ru}
\ and \ Irina Tsyganok\footnote{Department of Data Analysis and Financial Technologies,
Finance University, 49-55, Leningradsky Prospect, 125468 Moscow, Russia.
E-mail: i.i.tsyganok@mail.ru}
}

\begin{document}

\date{}

\maketitle

\begin{abstract}
In the paper, we study evolution equations of the scalar and Ricci curvatures  under the Hamilton's Ricci~flow on a closed  manifold and on a complete noncompact manifold.
In particular, we study conditions when the Ricci flow is trivial and the Ricci soliton is Ricci flat or Einstein.

\smallskip



\end{abstract}

\section{Introduction}

R.\,Hamilton introduced in \cite{36} an evolution equation method, which is called the Hamilton's Ricci flow, and proved that every closed
(i.e., compact without boundary) three-dimensional Riemannian manifold with positive Ricci curvature admits a metric of constant positive curvature. His work immediately led many mathematicians to study the Hamilton's Ricci flow and other evolution equations arising in differential geometry.
Self-similar solutions to the Hamilton's Ricci flow are called the Ricci solitons, they play an important role in the study of singularities of the flow. Over the past two decades, the geometry of Ricci solitons is the focus of attention of many~mathematicians  (see, for example, \cite{1,MT}).
In the present paper, we consider the evolution equations of the scalar and Ricci curvatures on a closed and complete noncompact manifolds under Hamilton's Ricci flow. In particular, we study conditions when the Hamilton's Ricci flow is trivial. For this, the concept of a stochastic complete metric is used.
We also consider applications of the obtained results on Hamilton's Ricci flow in the theory of Ricci solitons.
In~particular, we study conditions when a Ricci soliton is Einstein or trivial.  One of the attractive aspects of the work:
the results on the Hamilton's Ricci flow (Lemmas~\ref{T-01} and \ref{T-03}) have their analogues in the theory of Ricci solitons (Theorems~\ref{T-04} and \ref{T-06}).

\section{Evolution of the scalar curvature under the
%Hamilton's
Ricci~flow}
% under the Ricci flow}

Given a smooth manifold $M$ with a family $g(t)$ of Riemannian metrics, defined for a time interval $J\subset[0,\infty)$,
the \textit{Hamilton's Ricci flow} equation is
\begin{equation}\label{GrindEQ__1_}
 \frac{\partial}{\partial t}\,g(t)=-2\Ric_t.
\end{equation}
Here, $\Ric_t$ is the Ricci tensor of the metric $g(t)$.
It is well known that for any $C^{\infty}$-metric $g_{0}$ on a closed manifold $M$,
there exists a unique solution $g(t),\ t\in J=[0,\varepsilon)$, to the flow \eqref{GrindEQ__1_} for some $\varepsilon>0$
and $g(0)=g_{0}$ (see \cite{36}).
Moreover, if $(M,g_{0})$ is a closed Riemannian manifold with positive scalar curvature $s_{0}$ (a function on $M$),
then for a solution $g(t),\ t\in J$, to \eqref{GrindEQ__1_} with $g(0)=g_{0}$ we have (see \cite[p. 102]{1})
\[
 \varepsilon\le\frac{n}{2\,s_{\min}(0)}<\infty,
\]
where $s_{\min}(0)=\min\limits_{x\in M}s_0(x)$.
Under the Hamilton's Ricci flow, we have an \textit{evolution equation} for the scalar curvature $s(t)=\tr_{g(t)}\Ric_t$
%\ $(s(0)=s_0)$
in the following form (see \cite[p.~99]{1}):
\begin{equation}\label{GrindEQ__2_}
 \frac{\partial}{\partial t}\,s(t)=\Delta_{g(t)}\,s(t) + 2\,\|\Ric_t\|^2_{g(t)} .
\end{equation}
Here, $\Delta_{g(t)}=\tr_{g(t)}(\nabla^{g(t)})^{2}$ is the Laplace-Beltrami operator, defined as the trace of the~second covariant derivative for the Levi-Civita connection $\nabla^{g(t)}$ corresponding to the metric $g(t)$ (see also \cite[p. 17]{1}).
A standard scalar maximum principle argument applied to \eqref{GrindEQ__2_} proves that the minimum of the scalar curvature is a non-decreasing function of time $t$. In~addition, it shows that if $s_{\min}(0)>0$ then $s_{\min}(t)$ changes over time as follows
(see also \cite[p.~xxii]{MT}):
\[
 s_{\min}(t)\ge 1\big/\Big({\frac{n}{s_{\min}(0)}-2\,t}\Big)
\]
for $t\in J=[0,\,\varepsilon)$, where $\varepsilon\le\frac{n}{2\,s_{\min}(0)}$, and for a closed Riemannian manifold $(M, g_{0})$ with $s_0>0$. Therefore, $s_{\min}(t)>0$ and the following holds:

%\begin{proposition}
\textit{If $(M,g_{0})$ is a closed Riemannian manifold with positive scalar curvature, then for a~solution $g(t)$, $t\in[0,\; \varepsilon)$, to \eqref{GrindEQ__1_} with $g(0)=g_{0}$ and $s_{\min}(0)>0$ we have $s_{\min}(t)>0,\ t\in[0,\; \varepsilon)$}.
%\end{proposition}

Throughout the paper, we  consider \eqref{GrindEQ__1_} without requiring that the scalar curvature $s_0$ is positive.

By the following proposition, the scalar curvature
cannot be a decreasing (in average) function under a nontrivial Hamilton's Ricci flow on a closed Riemannian manifold.

% 1
\begin{lemma}\label{T-01}
Let $M$ be a smooth closed manifold with the Hamilton's Ricci flow \eqref{GrindEQ__1_} for a fami\-ly $g(t)$ of Riemannian metrics defined on a time interval $J\subset[0,\infty)$. If
the scalar curvature $s(t)$ for $t\in J$
is a \textit{decreasing in average} function under the
%Hamilton's Ricci
flow \eqref{GrindEQ__1_},
% on a closed manifold $M$,
that~is
\begin{equation}\label{GrindEQ__0_}
% \frac{\partial}{\partial t}\,s(t)\le 0,\quad t\in J\subset[0,\infty).
 \int_M\frac{\partial}{\partial t}\,s(t)\,d\vol_{g(t)}\le 0,\quad t\in J,
\end{equation}
then the flow is~trivial.
\end{lemma}

\begin{proof}
From \eqref{GrindEQ__2_} and \eqref{GrindEQ__0_} we obtain
\begin{equation}\label{GrindEQ__3_}
 \int_M\Delta_{g(t)}\,s(t)\,d\vol_{g(t)}\le \int_M\big[\Delta_{g(t)}\,s(t) -\frac{\partial}{\partial t}\,s(t)\big]\,d\vol_{g(t)}
 = -2\int_M\|\Ric_t\|^{2}_{g_0}\,d\vol_{g(t)} \le 0.
% \Delta_{g(t)}\,s(t)\le \Delta_{g(t)}\,s(t) -\frac{\partial}{\partial t}\,s(t) =-2\,\|\Ric_t\|^{2} \le 0.
\end{equation}
Since $\int_M\Delta_{g(t)}\,s(t)\,d\vol_{g(t)}=0$ (by Stokes's theorem),
then from \eqref{GrindEQ__3_} we conclude that $\Ric_t\equiv 0$, in particular, $s(t)\equiv 0$. In this case,
the Hamilton's Ricci flow \eqref{GrindEQ__1_} is trivial.
\end{proof}

Self-similar solutions to Hamilton's Ricci flow \eqref{GrindEQ__1_} are called the Ricci solitons, they evolve only by diffeomorphisms and scaling
(see \cite[p.~38]{MT}). Namely, consider a one-parameter family of diffeomorphisms
$\varphi:(t,x)\in\RR\times M\to\varphi_{t}(x)\in M$ that is generated by some smooth vector field $\xi$ on $M$.
The evolutive metric
\[
 g(t)=\sigma(t)\,\varphi_{t}^{*}(x)\,g(0)
\]
for a positive scalar $\sigma(t)$ such that $\sigma(0)=1$
and $g(0)=g_{0}$ is a {Ricci soliton} if and only if the metric $g_0$ is a solution of the nonlinear stationary PDE
\begin{equation} \label{GrindEQ__8_}
 -\Ric_0=\frac12\,{\cal L}_{\xi}\,g_0 +\lambda\,g_0 ,
\end{equation}
where $\Ric_0$ is the Ricci tensor of $g_0$, ${\cal L}_{\xi}\,g_0$ is the \textit{Lie derivative} of $g_0$ with respect to $\xi$ and $\lambda\in\RR$ (e.g., \cite[p.~38]{MT}). % \cite{1}. %[p. 22]{2}).
We denote the Ricci soliton in the following way: $(M, g_0,\,\xi,\,\lambda)$.
A~Ricci soliton $(M, g_0, \xi,\lambda)$ is called \textit{steady}, \textit{shrinking} and \textit{expanding}
if $\lambda=0,\ \lambda <0$ and $\lambda>0$, respectively.
In~addition, $(M, g_0,\,\xi{\rm},\,\lambda)$ is called \textit{Einstein} if ${\cal L}_{\xi}\,g_0=0$, and it is called \textit{trivial} if~$\xi\equiv0$.

% 3
\begin{remark}\rm
Any Ricci soliton $(M,g_0,\xi,\lambda)$ for a closed manifold $M$ is a fixed point of the Hamilton's Ricci flow \eqref{GrindEQ__1_} projected from the space of Riemannian metrics onto its quotient modulo diffeomorphisms and scalings,
and often arises as blow-up limits for the Hamilton's Ricci flow on a closed manifold.
\end{remark}

A \textit{gradient Ricci soliton} is a smooth
complete Riemannian manifold $(M,g_0)$ such that there exists a function $f\in C^\infty(M)$, sometimes called a \textit{potential function}, satisfying the condition $\xi=\nabla^{g_0} f=(df)^\sharp$, where $\sharp$ corresponds to raising indices of a tensor, therefore (see~\cite{11})
\begin{equation}\label{E-09}
 -\Ric_0 = {\rm Hess}_f +\lambda\,g_0 .
\end{equation}
In this case, we will use the notation $(M,g_0,f,\lambda)$.
By means of the Perel'man work \cite{P}, we have the following theorem (see also \cite[Theorem~3.1]{11}):
\textit{Every Ricci soliton $(M,g_0,\xi,\lambda)$ for a closed manifold $M$ is a gradient Ricci soliton}.

%4.
\begin{theorem}\label{T-04} Let $(M, g_0, f,\lambda)$ be a gradient Ricci soliton on a closed manifold $M$ with the following condition on the scalar curvature $s_0$ in the direction of a vector field $\xi=\nabla^{g_0} f$:
\begin{equation}\label{E-grad-s-f}
 \int_M \xi(s_0)\,d\vol_{g_0} \le 0 .
\end{equation}
Then $(M, g_0, f,\lambda)$ is an Einstein soliton.
\end{theorem}

\begin{proof}
From \eqref{E-09} we obtain $\bar\Delta f=s_0+n\,\lambda$ for the connection Laplacian $\bar\Delta=-\tr_{g_0}(\nabla^{g_0})^2$.
Since $\xi(s_0)=g_0(\nabla^{g_0} f, \nabla^{g_0} s_0)$, the following equalities hold:
\begin{equation}\label{Eq-nabla-fs}
 \int_M \xi(s_0)\,d\vol_{g_0} = -\int_M s_0\bar\Delta f\,d\vol_{g_0}
 =\int_M(\bar\Delta f - n\,\lambda)\bar\Delta f\,d\vol_{g_0} = \int_M(\bar\Delta f)^2\,d\vol_{g_0} \ge0.
\end{equation}
Then from \eqref{E-grad-s-f} and \eqref{Eq-nabla-fs} we obtain $\bar\Delta f=0$.
In this case, $f={\rm const}$ by the Bochner maximum principal (see \cite[p.~30--31]{12}).
Then, from \eqref{E-09} we conclude that $\Ric_0=-\lambda\,g_0=(s_0/n)\,g_0$,
i.e., $(M,g_0)$ is an Einstein manifold.
\end{proof}

The length $\|\nabla f\|$ of the gradient of a scalar function $f$ on $M$ characterizes the rate of change of $f$.
In our case, the scalar curvature $s_0$ is ``decreasing in average" in the direction of $\xi=\nabla f$.

% 4
\begin{remark}\rm
In \cite{C} it was proved that any closed steady or expanding Ricci soliton is Einstein. Therefore, Theorem~\ref{T-04} concerns only a shrinking Ricci soliton. Moreover, Theorem~\ref{T-04} can be considered as a consequence of Lemma~\ref{T-01}.
Note that the condition \eqref{E-grad-s-f} in Theorem~\ref{T-04} can be replaced by the more strict condition $s_0 = {\rm const}$. Therefore, a shrinking Ricci soliton with constant scalar curvature on a closed manifold is an Einstein soliton.
\end{remark}

%%%%%%%%%%%%%%%%
Besides studying the Hamilton's Ricci flow on a closed manifold, we shall consider the flow \eqref{GrindEQ__1_} on an open (i.e., noncompact complete) manifold.
Namely, a solution $g(t),\ t\in J\subset[0,\infty)$, of
\eqref{GrindEQ__1_} will be called \textit{complete} if for each $t\in J$ the Riemannian metric $g(t)$ is complete (see \cite[p.~102]{1}). On the other hand, a solution $g(t),\ t\in J$, of
\eqref{GrindEQ__1_} will be called \textit{stochastically complete} if for each $t\in J$ the metric $g(t)$ is stochastically complete.
Let us dwell in more detail on the concept on stochastically complete Riemannian manifolds (see \cite{G1,G2,G3}).

Let $(M, g(t))$ be a connected, open Riemannian manifold for some $t\in J\subset[0,\infty)$.
Consider a diffusion process on $(M, g(t))$, generated by the Laplace-Beltrami operator $\Delta_{g(t)}$, which is absorbing at infinity. If the probability of the absorption at $\infty $ in a finite time is zero, then $(M,g(t))$ is called \textit{stochastically complete}.
The main result of \cite{G2} is the following:
\textit{Let $(M,g(t))$ be geodesically complete and
 $\int_{0}^{\infty}\big({R/\ln V_{R}(t)}\big)\,dR=\infty$,
where $V_{R}(t)$ is the $g(t)$-volume of the geodesic sphere of radius $R$ and fixed center. Then $(M,g(t))$ is stochastically complete}.

% 1
\begin{remark}\rm
There are other criteria for the stochastic completeness of a manifold (e.g., \cite{PRS}).
We~refer the reader to the excellent review \cite{G3} for an exhaustive exposition of the theory of stochastically complete manifolds.
\end{remark}

\begin{theorem}
%\label{T-02}
Let $M$ be a smooth manifold with the Hamilton's Ricci flow \eqref{GrindEQ__1_} for a family $g(t)$ of stochasti\-cally complete Riemannian metrics defined on a time interval $J\subset[0,\infty)$.
If for each $t\in J$, the scalar curvature $s(t)$ is a nonnegative function from $L^{1}(M, g(t))$ and
\begin{equation}\label{Eq-3b}
 \frac{\partial}{\partial t}\,s(t)\le 0,\quad t\in J,
\end{equation}
holds then the flow is trivial.
\end{theorem}

\begin{proof} By \eqref{GrindEQ__3_}, $s(t)$ is a \textit{superharmonic function} (see \cite[p.~233]{G1}).
In turn, if $(M, g(t)),\ t\in J$, is stochastically complete then any nonnegative superharmonic function
$u(t)\in L^{1}(M, g(t))$ is a constant (see \cite{G2}; \cite[p. 234]{G1}). In particular, the nonnegative scalar function $s(t)\in L^{1}(M, g(t))$ is constant, thus $\Delta_{g(t)}\,s(t)\equiv0$, and from \eqref{GrindEQ__3_} obtain $\Ric_t\equiv 0$ for $t\in J$. In this case, the flow \eqref{GrindEQ__1_} is trivial.
\end{proof}

In turn, S.\,T.~Yau proved that a complete Riemannian manifold with the Ricci curvature bounded below is stochastically complete (see \cite{Y}). Then we have the following

\begin{corollary}
Let $M$ be a smooth manifold with the Hamilton's Ricci flow \eqref{GrindEQ__1_} for a family $g(t)$ of complete Riemannian metrics defined on a time interval $J\subset[0,\infty)$.
If for each $t\in J$, the Ricci tensor $\Ric_t$ is bounded below, the scalar curvature $s(t)$ is a nonnegative function from $L^{1}(M,g(t))$ and \eqref{Eq-3b} holds, then the flow is trivial.
\end{corollary}

Recall that the \textit{kinetic energy} ${\cal E}(\xi)$ of the flow on a Riemannian manifold $(M,g)$ gene\-rated by a vector field $\xi$ is given, in accordance with \cite[p. 2]{AK}, by
\begin{equation*}
 {\cal E}(\xi)= \frac12\int_M \|\xi\|^2\,d\vol_g,
\end{equation*}
where $\|\xi\|^2=g(\xi,\xi)$. The energy ${\cal E}(\xi)$ can be infinite or finite. For example, ${\cal E}(\xi)<\infty$ for a smooth complete vector field $\xi$ on a closed manifold $(M,g)$.
Recall that a vector field $\xi$ on a Riemannian manifold $(M, g)$ is an \textit{infinitesimal harmonic transformation} if the flow of $\xi$ consists of harmonic diffeomorphisms of $(M, g)$ (see~\cite{10}).
The Laplace-Beltrami operator $\Delta(\|\xi\|^2)$ for an infinitesimal harmonic transformation $\xi$ has the form (see \cite{4,21,22,SM})
\begin{equation}\label{Eq-011}
 (1/2)\Delta(\|\xi\|^2) = \|\nabla\xi\|^2 -\Ric(\xi,\xi).
\end{equation}
In particular, for an arbitrary Ricci soliton $(M, g,\,\xi,\,\lambda)$ its vector field $\xi$ is an infinitesimal harmonic transformation (see~\cite{3}). Therefore, the equation \eqref{Eq-011} is also valid for the vector field $\xi$ of a Ricci soliton  $(M, g,\,\xi,\,\lambda)$.
We use the second Kato inequality (see \cite[p.~380]{B})
\begin{equation}\label{Eq-012}
 -\|\xi\|\,\Delta(\|\xi\|) \le g(\bar\Delta\theta,\theta),
\end{equation}
where $\theta$ is the 1-form $g$-dual to $\xi$. Recall that the Yano Laplacian has the form (see \cite[p. 40]{YK})
\[
 \square\,\theta = \bar\Delta\,\theta -\Ric(\xi,\,\cdot),
\]
therefore, $\bar\Delta\theta=\Ric(\xi,\,\cdot)$ for the vector field $\xi$ of a Ricci soliton $(M, g,\,\xi,\,\lambda)$.
Then, from \eqref{Eq-012} we get the inequality
\begin{equation}\label{Eq-013}
 \|\xi\|\,\Delta(\|\xi\|) \ge -\Ric(\xi,\xi).
\end{equation}

%4
\begin{theorem}
%\label{T-05}
Let $(M, g_0, \xi, \lambda)$ be a complete Ricci soliton with the Ricci curvature $\Ric_0(\xi,\xi)\le0$ and finite kinetic energy ${\cal E}(\xi)$,
then $(M, g_0)$ is Einstein.
In addition,

$\bullet$\
if $(M, g_0,\,\xi,\,\lambda)$ is shrinking or expanding then it is trivial.

$\bullet$\
if $(M, g_0,\,\xi,\,\lambda)$ is steady and ${\rm Vol}(M,g_0)=\infty$ then it is~trivial.
\end{theorem}

\begin{proof}
Using \cite[Theorem~3]{18}, we find from \eqref{Eq-013} that if $\Ric_0(\xi,\xi)\le0$ and ${\cal E}(\xi)<\infty$ then $\|\xi\|_{g_0}=C$ for some constant $C\ge0$, and hence $\Ric_0(\xi,\xi)=0$. In this case, from \eqref{Eq-011} we obtain $\nabla\xi=0$. Therefore, we can conclude from \eqref{GrindEQ__8_} that $\Ric_0=-\lambda\,g_0$,
hence the Ricci soliton $(M, g_0,\,\xi,\,\lambda)$ is Einstein. In this case, $0=\Ric_0(\xi,\xi)=-\lambda\,\|\xi\|^2_{g_0}$.
Therefore, if $\lambda\ne0$ then $\xi=0$, and thus the Ricci soliton $(M, g_0,\,\xi,\,\lambda)$ is trivial.
At the same time, if $\lambda=0$ and ${\rm Vol}(M,g_0)=\infty$ then we get the equality
\[
 {\cal E}(\xi)
 %= \frac12\int_M \|\xi\|^2\,d\vol_g
 = \frac12\,C^2\int_M d\vol_{g_0} =\infty ,
\]
which contradicts the assumption ${\cal E}(\xi)<\infty$. In this case, $C=0$; hence, the Ricci soliton
$(M, g_0,\,\xi,\,\lambda)$ is trivial.
\end{proof}

\section{Evolution of the Ricci tensor under the Ricci~flow}

 The evolution of the Ricci tensor under the flow \eqref{GrindEQ__1_}  has the form (see \cite[p. 112]{1})
\begin{equation}\label{GrindEQ__4_}
 \frac{\partial}{\partial t} \Ric_t=\Delta_{g(t)}^{S}\Ric_t,
\end{equation}
where $\Delta_{g(t)}^{S} : C^{\infty}(S^{2} M)\to C^{\infty}(S^{2} M)$ is the \textit{Sampson Laplacian} (see \cite{10,SM}).
The Laplacian
%\[
 $\Delta_{g(t)}^{S}=\delta\,\delta^*-\delta^*\,\delta$
%\]
acts on the vector bundle $S^pM$ of covariant symmetric $p$-tensors for $p\ge1$, where $\delta^*=(p+1){\rm Sym}\nabla^{g(t)}$
with the ordinary operator of symmetrization ${\rm Sym}:\bigotimes^p\,T^*M\to S^pM$, and $\delta:C^\infty(S^{p+1}M)\to C^\infty(S^pM)$ is the divergence operator (see \cite[p.~55; 356]{Besse}). Moreover,
$\Delta^S_{g(t)}$ admits the Weitzenb\"{o}ck decomposition
\[
 \Delta_{g(t)}^{S}=\bar\Delta_{g(t)}-\Re_p,
\]
where $\Re_p$ is the Weitzenb\"{o}ck curvature operator of the Lichnerowicz Laplacian $\Delta^L_{g(t)}=\bar\Delta_{g(t)}+\Re_p$ restricted to covariant symmetric $p$-tensors.
For example, if $\varphi_{ij}=\varphi(e_i,e_j)$ are local components of an arbitrary $\varphi\in C^\infty(S^2M)$ with respect to an orthonormal basis $\{e_1,\ldots,e_n\}$ of $T_xM$ at a point $x\in M$, then
\[
 \Re_2(\varphi)=-2\,\varphi_{kl}R_{kijl}(t)+\varphi_{kj}R_{ki}(t)+\varphi_{ki}R_{kj}(t)
\]
for the local components $R_{kijl}(t)$ and $R_{ki}(t)$ of the curvature tensor $Rm(t)$ and the Ricci tensor $\Ric_t$.
In~addition, the identity $\tr_{g(t)}(\Delta_{g(t)}^{S}\,\varphi)=\Delta_{g(t)}^{S}(\tr_{g(t)}\varphi)$ holds for any $\varphi\in C^\infty(S^pM)$.

\begin{lemma}\label{T-03}
Let $M$ be a closed smooth manifold with the Hamilton's Ricci flow \eqref{GrindEQ__1_} for a family $g(t)$ of Riemannian metrics, defined on a time interval $J\subset[0,\infty)$. If~the inequality
\begin{subequations}
\begin{equation}\label{Eq-02a}
 \tr_{g(t)}\Big(\frac{\partial}{\partial t}\,\Ric_t\Big) \ge 0
\end{equation}
or
\begin{equation}\label{Eq-02b}
 \tr_{g(t)}\Big(\frac{\partial}{\partial t}\,\Ric_t\Big) \le 0,
\end{equation}
is satisfied for each $t\in J$, then the flow is trivial.
\end{subequations}
\end{lemma}

\begin{proof}
From \eqref{GrindEQ__4_} we obtain
\begin{equation}\label{GrindEQ__5_}
 \tr_{g(t)}\,\Big(\frac{\partial}{\partial t}\Ric_t\Big)=\bar\Delta_{g(t)}\, s(t),
\end{equation}
because $\tr_{g(t)}(\Delta_{g(t)}^{S}\,\varphi) = \bar\Delta_{g(t)}(\tr_{g(t)}\varphi)$ for any $\varphi\in C^{\infty}(S^{2} M)$.
If we suppose \eqref{Eq-02a} or \eqref{Eq-02b}
for each $t\in J$, then by \eqref{GrindEQ__5_} we obtain
$\bar\Delta_{g(t)}\,s(t)\ge 0$ or $\bar\Delta_{g(t)}\,s(t)\le 0$, respectively.
For a closed manifold $M$, by the \textit{Bochner maximum principle} (see \cite[p.~30--33]{5}; \cite[p.~39]{YK}) we find $s={\rm const}$.
In~this case, the flow \eqref{GrindEQ__1_} is trivial.
\end{proof}

Next, if we assume the inequality \eqref{Eq-02a} then from \eqref{GrindEQ__5_} we obtain $\bar\Delta_{g(t)}\,s(t)\ge 0$.
In this case, $\Delta_{g(t)}\,s(t)=-\bar\Delta_{g(t)}\,s(t)\le 0$ according to definitions of Bochner and Laplace-Beltrami Laplacians; therefore, $s(t)$ is a superharmonic function.

\begin{proposition}
Let $M$ be a smooth manifold with the Hamilton's Ricci flow \eqref{GrindEQ__1_} for a family $g(t)$ of stochastically complete Riemannian metrics defined on a time interval $J\subset[0,\infty)$. If~the inequality \eqref{Eq-02a} is satisfied for each $t\in J$ and the scalar curvature $s(t)$ is a nonnegative function from $L^{1}(M,g(t))$ then the flow is trivial.
\end{proposition}

\begin{proposition}\label{P-02}
Let $M$ be a smooth manifold with the Hamilton's Ricci flow \eqref{GrindEQ__1_} for a family $g(t)$ of complete Riemannian metrics defined on a time interval $J\subset[0,\infty)$. If for each $t\in J$, the Ricci tensor $\Ric_t$ is bounded below,  the inequality \eqref{Eq-02a} is satisfied and the scalar curvature $s(t)$ is a nonnegative function from $L^{1}(M,g(t))$ then the flow is trivial.
\end{proposition}

On the other hand, let the Riemannian metric $g(t)$ be only complete for each $t\in J\subset[0,\infty)$.
From the inequality \eqref{Eq-02b} we obtain that $\Delta_{g(t)}\,s(t)\ge 0$, i.e., $s(t)$ is a \textit{subharmonic function}.
S.\,T.~Yau proved in \cite[p. 663]{18} the following:
\textit{Let $u$ be a nonnegative subharmonic function on a complete manifold $(M,g)$, then $\int_{M} u^{q}\,d\vol_{g}=\infty$ for $q>1$, unless $u$ is constant}.

\smallskip

In this case, we obtain the following consequence of Proposition~\ref{P-02}.

\begin{corollary} Let $M$ be a smooth manifold with the Hamilton's Ricci flow \eqref{GrindEQ__1_} for a family $g(t)$ of complete Riemannian metrics defined on a time interval $J\subset[0,\infty)$. If the inequality \eqref{Eq-02b} is satisfied for each $t\in J$ and the scalar curvature $s(t)$ is a nonnegative function from $L^{q}(M,g(t))$ for some $q>1$ then the flow is~trivial.
\end{corollary}

Let $(M, g_0,\,\xi,\,\lambda)$ be a Ricci soliton on a closed manifold $M$. In \cite{21}, the following formula
was proved: $\bar\Delta_{g_0}\Div\xi=\tr_{g_0}({\cal L}_\xi\Ric_0)$  for an arbitrary infinitesimal harmonic transformation~$\xi$.
This equation can be rewritten in the form
\begin{equation}\label{Eq-14}
 \bar\Delta_{g_0}\,s_0 = \tr_{g(0)}({\cal L}_\xi\Ric_0)
\end{equation}
for the case of a Ricci soliton $(M, g_0,\,\xi,\,\lambda)$. If $\tr_{g(0)}({\cal L}_\xi\Ric_0)\ge0$ (or $\tr_{g(0)}({\cal L}_\xi\Ric_0)\le0$) then we obtain $s_0={\rm const}$ from \eqref{Eq-14} by the Bochner maximum principal (see \cite[p. 30--33]{12}; \cite[p. 39]{YK}).
In~this case, by the Green's formula $\int_M(\Div\xi)\,d\vol_{g_0}=0$ we get $\Div\xi=0$, because $\Div\xi=s_0+n\,\lambda={\rm const}$.
On the other hand, the following two equations:
\[
 \Div\xi=0,\quad
 \square\,\theta=0,
\]
where $\theta^\sharp=\xi$, characterize the vector field $\xi$ as an infinitesimal isometry (see \cite[p. 44]{YK}).
For~this vector field $\xi$ we have ${\cal L}_\xi\,g_0=0$, hence $\Ric_0=-\lambda\,g_0$.

 The following theorem can be considered as a consequence of Lemma~\ref{T-03}.

\begin{theorem}\label{T-06}
Let $(M, g_0,\,\xi,\,\lambda)$ be a closed Ricci soliton. If either $\,\tr_{g_0}({\cal L}_\xi\Ric_0)\ge0$ or $\,\tr_{g_0}({\cal L}_\xi\Ric_0)\le0$ then it is an Einstein soliton.
\end{theorem}

From \eqref{Eq-14} we find $\Delta_{g_0}\,s_0 = -\tr_{g(0)}({\cal L}_\xi\Ric_0)$.
Then from $\tr_{g(0)}({\cal L}_\xi\Ric_0)\le0$ we obtain $\Delta_{g_0}\,s_0\ge0$,
i.e., $s_0$ is a subharmonic function.
In this case, if the scalar curvature $s_0$ is a nonnegative function from $L^q(M,g_0)$ for some $q>1$ then $s_0$ is constant.
On the other hand, if $\tr_{g(0)}({\cal L}_\xi\Ric_0)\ge0$ holds then $\Delta\,s_0\le0$, i.e., $s_0$ is a superharmonic function.

In turn, if $(M,g_0)$ is a stochastically complete manifold and $s_0$ is a nonnegative function from  $L^1(M,g_0)$ then $s_0$ is constant. Thus, in both cases we have
\[
 \nabla s_0=0,\quad \tr_{g(0)}({\cal L}_\xi\Ric_0)=0.
\]
From the above, we conclude that $g_0(\Ric_0,{\cal L}_\xi\Ric_0)=0$. Then, from \eqref{GrindEQ__8_} we get $\|\Ric_0\|^2=-\lambda\,s_0$, where $s_0$ is a nonnegative constant. Therefore, if $\lambda>0$ or $\lambda=0$ then $\Ric_0=0$, hence $(M, g_0,\,\xi,\,\lambda)$ is Ricci flat. Moreover, if $s_0=0$ then $\Ric_0=0$, so $(M, g_0,\,\xi,\,\lambda)$ is also Ricci~flat. Therefore, we can formulate the following

\begin{theorem}\label{T-06b}
Let $(M, g_0,\,\xi,\,\lambda)$ be a steady or expanding Ricci soliton with nonnegative scalar curvature $s_0$.
If either $(M, g_0)$ is complete with $s_0\in L^q(M,g_0)$ for some $q>1$ and $\tr_{g(0)}({\cal L}_\xi\Ric_0)\le0$,
or $(M, g_0)$ is stochastically complete with $s_0\in L^1(M,g_0)$ and $\tr_{g(0)}({\cal L}_\xi\Ric_0)\ge0$,
then $(M, g_0)$ is Ricci flat.
\end{theorem}

\begin{remark}\rm Recall \cite{GR} that a complete noncompact gradient Ricci soliton $(M, g_0,\,f,\,\lambda)$ is stochastically complete. Moreover, a complete, expanding gradient Ricci soliton with nonnega\-tive scalar curvature $s_0$ from $L^1(M,g_0)$ is isometric to Euclidean space, see \cite{PRS-new}. Thus, Theorem~\ref{T-06b} is an analogue of this statement for a non-gradient, expanding Ricci~soliton.
\end{remark}

Let $(M, g_0,\,f,\,\lambda)$ be a complete gradient shrinking Ricci soliton, then its scalar curvature $s_0$ is nonnegative (see \cite{C2}).
In addition, if $(M, g_0,\,f,\,\lambda)$ is noncompact then ${\rm Vol}(M,g_0)=\infty$ (see~\cite{CZ}).
In this case, if $\tr_{g_0}({\cal L}_\xi\Ric_0)\le0$ for $\xi=\nabla^{g_0}f$ and $s_0\in L^q(M,g_0)$ for some $q>1$, then $s_0$ is constant and $\|\Ric\|^2_{g_0}=-\lambda\,s_0$ (see the proof of Theorem~\ref{T-06b}).
At the same time, we have
\[
 \int_M (s_0)^q\,d\vol_{g_0}=(s_0)^q\,{\rm Vol}(M,g_0)=\infty,
\]
which contradicts to the assumption $s_0\in L^q(M,g_0)$ for some $q>1$. In this case, $s_0=0$, thus $\Ric_0=0$, so $(M, g_0,\,f,\,\lambda)$ is Ricci flat. Therefore, the following theorem holds.

\begin{theorem}
%\label{T-06c}
Let $(M, g_0,\,f,\,\lambda)$ be a complete noncompact gradient shrinking Ricci soliton.
If~its scalar curvature $s_0\in L^q(M,g_0)$ for some $q>1$ and $\tr_{g_0}({\cal L}_\xi\Ric_0)\le0$ then $(M, g_0,\,f,\,\lambda)$ is Ricci flat.
\end{theorem}

\section{Evolution of the Riemannian curvature tensor under the Ricci flow}

The evolution equations of the Riemannian curvature tensor $Rm$ under the Hamilton's flow \eqref{GrindEQ__1_} has the form
(see \cite[p.~119]{1})
\begin{eqnarray*}
 && \frac{\partial }{\partial t}\,R_{ijkl}(t) = \Delta_{g(t)}\,R_{ijkl}(t) +2\,( B_{ijkl}(t) - B_{ijlk}(t) + B_{ikjl}(t) -B_{iljk}(t) )\\
 &&\hskip-5mm
 - g^{pq}(t) R_{ip}(t) R_{qjkl}(t)
 - g^{pq}(t)R_{jp}(t) R_{iqkl}(t)
 - g^{pq}(t) R_{kp}(t) R_{ijql}(t)
 - g^{pq}(t) R_{lp}(t) R_{ijkp}(t) ,
\end{eqnarray*}
where $B_{ijkl}(t) {=}-g^{pr}(t) g^{qm}(t) R_{ipjq}(t) R_{krlm}(t)$,
$R_{ijkl}(t)$ are local components of $Rm(t)$,
$R_{ij}(t)$ are local components of $\Ric(t)$ and $(g^{ij}(t))=g^{-1}(t)$.
It is not difficult to establish that
\begin{equation}\label{E-ineq1}
 g^{jk}(t) g^{il}(t)\frac{\partial}{\partial t}\,R_{ijkl}(t) \ge 0
\end{equation}
Then from the inequality
$g^{jk}(t) g^{il}(t)\frac{\partial}{\partial t}\,R_{ijkl}(t) \ge 0$
we obtain $\Delta_{g(t)}\,s(t)\ge 0$.
For a closed manifold $M$, by the Bochner maximum principle (see [9, p. 30--33]; [26, p. 39]) we find $s(t)={\rm const}$, hence, $\Ric(t)=0$.
If the inequality \eqref{E-ineq1} is satisfied for each $t\in J$, then the flow \eqref{GrindEQ__1_} is trivial.
%In this case, the following proposition holds.

%3
\begin{lemma}\label{L-03} Let $M$ be a closed smooth manifold with the Hamilton's Ricci flow \eqref{GrindEQ__1_} for a family $g(t)$ of Riemannian metrics, defined on a time interval $J\subset[0,\infty)$. If the inequality \eqref{E-ineq1} is satisfied for each $t\in J$,
%where $R_{ijkl}(t)$ are local components of $R(t)$ and $(g^{ij}(t))=g(t)^{-1}$,
then the flow \eqref{GrindEQ__1_} is trivial.
\end{lemma}

On the other hand, let $M$ be a noncompact smooth manifold. Consider a family $g(t)$ of Riemannian complete metrics, defined on a time interval $J\subset[0,\infty)$. If the inequality \eqref{E-ineq1} is satisfied for each $t\in J$,
then we can conclude that $s(t)$ is subharmonic function for each $t\in J$. 
Due to Yau theorem (see \cite[p.~663]{18}) we formulate the following

% 7
\begin{theorem} Let M be a noncompact smooth manifold with the Hamilton's Ricci flow \eqref{GrindEQ__1_} for a family $g(t)$ of Riemannian complete metrics, defined on a time interval $J\subset[0,\infty)$. Let the scalar curvature $s(t)$ of $g(t)$ is nonnegative function from $L^{q}(M, g(t))$ for some $q>1$ and each $t\in J$. In addition, if the inequality \eqref{E-ineq1} is satisfied for each $t\in J$,
%local components $Rm(t)_{ijkl}$ of the Riemannian curvature tensor $R(t)$ of $g(t)$ and $(g^{ij}(t))=g(t)^{-1}$,
then the flow \eqref{GrindEQ__1_} is trivial.
\end{theorem}

It is well-known that
\[
 {\cal L}_{\xi} R_{ijkl}
 =\xi^{m} \nabla_{m} R_{ijkl} +\nabla_{i}\,\xi^{m} R_{mjkl} +\nabla_{j}\,\xi^{m} R_{imkl} +\nabla_{k}\,\xi^{m} R_{ijml} +\nabla_{l}\,\xi^{m} R_{ijkm},
\]
where ${\cal L}_{\xi}$ is the Lie derivative with respect to an arbitrary smooth vector field $\xi$.
%, $R_{ijkl}$ are local components of $Rm$, which is determined by the metric $g_{0}$.
It is not difficult to establish that
\[
 g_{0}^{jk} g_{0}^{il} {\cal L}_{\xi } R_{ijkl} =\xi \left(s_{0} \right)+4\, g_{0}^{jk} g_{0}^{il} R_{ij} \nabla_{k} \xi_{l} ,
\]
where $R_{ij}$
%are local components $\Ric$, which
are determined by the metric $g_{0} $ and $(g_{0}^{ij})=g_{0}^{-1}$.
In the case of a closed manifold $M$ we have
\[
 \int_{M}\big(g_{0}^{jk} g_{0}^{il} {\cal L}_{\xi } R_{ijkl}\big)\,d\vol_{g_{0}}
 =\int_{M}\xi(s_{0})\,d\vol_{g_{0}} + 4\int_{M}\big(g_{0}^{jk} g_{0}^{il} R_{ij} \nabla_{k} \xi_{l} \big)\,d\vol_{g_{0}} .
\]
In turn, using the following integral formulas:
\begin{eqnarray*}
 \int_{M}\nabla_{i}\big(g_{0}^{jk} R_{ki} \xi^{i} \big)\, d\vol_{g_{0}}
 &=& \int_{M}\big(g_{0}^{jk} \nabla_{j} R_{ki} \big)\, \xi^{i} \, d\vol_{g_{0}}
 + \int_{M}\big(g_{0}^{jk} g_{0}^{il} R_{ij} \nabla_{k} \xi_{l} \big)\,d\vol_{g_{0}} \\
 &=& \frac{1}{2}\int_{M}\xi(s_{0})\,d\vol_{g_{0}}
 +\int_{M}\big(g_{0}^{jk} g_{0}^{il} R_{ij} \nabla_{k} \xi_{l} \big)\,d\vol_{g_{0}},
\end{eqnarray*}
we conclude that
\[
 \int_{M}\big(g_{0}^{jk} g_{0}^{il} {\cal L}_{\xi} R_{ijkl} \big)\,d\vol_{g_{0}} =-\int_{M}\xi(s_{0})\,d\vol_{g_{0}} .
\]
Then from the inequality
\begin{equation}\label{E-ineq}
 \int_{M}\big(g_{0}^{jk} g_{0}^{il} {\cal L}_{\xi } R_{ijkl} \big)\,d\vol_{g_{0}} \ge 0
\end{equation}
we obtain $\int_{M}\xi (s_{0} )\,d\vol_{g_{0}}\le 0$. Now, consider the Ricci soliton $(M, g_{0}, \xi, \lambda)$.
In our case, $(M, g_{0}, \xi, \lambda)$ must be an Einstein soliton (see \cite[Theorem~3.1]{11} and our Theorem~\ref{T-04}).
As the result, we formulate the following theorem.

% 8
\begin{theorem}\label{T-08}
Let $(M, g_{0}, \xi, \lambda)$ be a Ricci soliton on a closed manifold M with condition \eqref{E-ineq} on the Riemannian curvature tensor $Rm$.
%\[
% \int_{M}\big(g_{0}^{ik} g_{0}^{jl} {\cal L}_{\xi} R_{ijkl} \big)\,d\vol_{g_{0}}\ge 0 .
%\]
%where $R_{ijkl}$ are local components of $R$, ${\cal L}_{\xi}$ is the Lie derivative with respect to $\xi$ and $(g_{0}^{ij})=g_{0}^{-1}$.
Then $(M, g_{0}, \xi, \lambda)$ is an Einstein soliton.
\end{theorem}

\begin{remark}\rm In \cite{C} it was proved that any closed steady or expanding Ricci soliton is Einstein.
Therefore, Theorem~\ref{T-08} concerns only a shrinking Ricci soliton. Moreover, Theorem~\ref{T-04} can be considered as a consequence of Lemma~\ref{L-03}.
\end{remark}

\end{document}